\documentclass[12pt,a4paper,reqno]{amsart}

\usepackage[active]{srcltx}
\usepackage{enumitem}
\usepackage{float}
\usepackage{amssymb}

\usepackage{lineno,hyperref}
\usepackage[active]{srcltx}
\usepackage{float}
\usepackage{mathtools}
\usepackage[all]{xy}

\usepackage{xcolor}

\renewcommand{\dim}{\operatorname{dim}}
\newcommand{\Aut}{\operatorname{Aut}}

\newcommand{\Z}{{\mathbb Z}}
\newcommand{\M}{{\mathcal M}}
\newcommand{\Q}{{\mathbb Q}}
\newcommand{\C}{{\mathbb C}}

\newcommand{\E}{{\mathcal E}}

\newcommand{\Or}{{\operatorname O}}
\newcommand{\di}{{\operatorname d}}

\newtheorem{theorem}{Theorem}[section]
\newtheorem{lemma}[theorem]{Lemma}
\newtheorem{corollary}[theorem]{Corollary}
\newtheorem{proposition}[theorem]{Proposition}

\newtheorem{remark}[theorem]{Remark}
\newtheorem{example}[theorem]{Example}

\newtheorem{definition}[theorem]{Definition}

\newtheorem{question}[theorem]{Question}

\title{{On the realizability of group actions}}

\author{Cristina ~Costoya}
\address[C.~Costoya]{Departamento de Computaci\'on, \'Alxebra,
Universidade da Coru{\~n}a, Campus de Elvi{\~n}a, 15071  A Coru{\~n}a, Spain.}
\email[C.~Costoya]{cristina.costoya@udc.es}

\author{Antonio ~Viruel}
\address[A.~Viruel]{
Departamento de {\'A}lgebra, Geometr{\'\i}a y Topolog{\'\i}a,
Universidad de M{\'a}\-la\-ga, Campus de Teatinos, 29071 M{\'a}laga,
Spain.}
\email[A.~Viruel]{viruel@uma.es}

\thanks{First author is partially  supported by Ministerio de Econom\'ia y Competitividad grant 
MTM2013-43687-P (European FEDER support included). Second author is partially supported by Ministerio de Econom\'ia y Competitividad grant 
MTM2013-41768-P (European FEDER support included). Both authors are partially supported by Xunta de Galicia grant EM 2013/016.}

\begin{document}

\begin{abstract}

We raise the question of realizability of group actions which is an extended version of the 1960's Kahn  realizability problem for (abstract) groups. Namely, if $M$ is a  $\mathbb ZG$-module for a group $G$, we say that  a simply-connected space $X$ realize this action if,  for some $k$,  $\pi_k(X)$ as a $\mathbb Z \mathcal E (X) $-module for the group $\mathcal E (X)$ of self-homotopy equivalences of $X$,  is isomorphic to  $M$ as a  $\mathbb ZG$-module. Which modules can be so realized? In this paper we obtain a positive answer for any faithful finitely generated $\Q G$-module, where $G$ is finite. Our proof relies on providing a positive answer to Kahn's problem for a large class of orthogonal groups of which,  by using invariant theory, our case is shown to be a particular one.

\end{abstract}

\maketitle

\section{Introduction}\label{intro}

Realizability problems in algebraic topology are very easy to state and extremely difficult to solve.  Classical examples of this kind are
the \emph{realizability of cohomological algebras} settled by N. E. Steenrod in the 1960's, which asks for characterization of graded algebras that appear as the cohomology algebra of a space, \cite{Aguade}, \cite{KasperJesper}; the \emph{$G$-Moore space problem},  also by Steenrod, which asks for  characterization of $\Z G$-modules that appear as the homology of some simply-connected $G$-Moore space $X$, \cite[Problem 51]{Lashof}, \cite{GC};  the \emph{realizability problem for abstract groups} proposed by D. Kahn, which asks for characterization of groups that appear as the group of self-homotopy equivalences of a simply-connected spaces, \cite{ka1}, \cite{CV2}.
They all ask which algebraic structures occur as a given homotopy-type of a space.

In this paper we present a question which is an extension, and  an homotopic dual, respectively, of the Kahn realizability problem and the {$G$-Moore space problem}  presented above.  To be more precise, we say that a $\Z G$-module $M$ is realized by a simply-connected space $X$ if it occurs as the $\Z \E (X)$-module $\pi_k(X)$ for some $k$, where $\E (X)$  is the group of self-homotopy equivalences with its natural action (by composition) on the homotopy groups $\pi_k(X)$, for every $k\geq 2$.

Then, one might ask the following question:
\begin{question}[Realizability of group actions]\label{question} Which finetely generated  $\Z G$-modules can be realized by simply-connected spaces $X$?
\end{question}

Restrictions on this question are studied here, mainly because realizing a group action implies in particular to realize (in the Kahn sense) the group  itself, which is still an open problem in the general case.  Key progress was made though in a recent  work \cite{CV2}, where we proved that every finite group  is the group of self homotopy equivalences of a simply-connected rational space.  The use of rational homotopy theory techniques and the ingenious result of Frucht \cite{Frucht1} that expresses every finite group as the automorphism group of a finite graph, allowed us to realize any finite group through a minimal Sullivan algebra of finite type (a rational space) encoding the \emph{finite} set of vertices and edges of a graph.  Here, we exploit the same techniques of rational homotopy theory, though we follow a significantly different approach that will be carried out in Section \ref{invariant}. Namely, if the action is faithful over a finitely generated $\Q G$-module $M$, and $G$ is finite, we make a connection with Invariant Theory to represent $G \leq {\rm{GL}} (M)$ as $\Or (\mathcal Q) \leq {\rm{GL}} (M)$, the inclusion of the orthogonal group of a family $\mathcal Q$ of algebraic forms over $M$  in the general linear group, Theorem \ref{thm:orthogonal}. Then, in Section \ref{Model}, we show that a large class of orthogonal groups $\Or  (\mathcal Q) $ can be realized in the Kahn sense by a minimal Sullivan algebra encoding the algebraic forms, Theorem  \ref{thmgroup} and Corollary \ref{diagonal}.  The general result that we prove here and that we obtain  as an application of Theorem  \ref{thmgroup} and Theorem \ref{thm:orthogonal} above mentioned is as follows.

 \begin{theorem}\label{main} Every faithful and finitely generated $\mathbb QG$-module $M$ where $G$ is a finite group,  can be realized by infinitely many (non-homotopy equivalent) rational spaces $X$.
\end{theorem}
In the previous theorem,  spaces $X$ are Postnikov pieces, or in other words, $G$-objects in the homotopy category of topological spaces and,  the module $M$ appears as $\pi_{40} (X)$ whereas the $G$-action on $\pi_{\ne 40} (X)$ is trivial.  In that sense, Question \ref{question} is not only a generalization of Kahn's problem, but an homotopic dual of the $G$-Moore problem aforementioned. Remark that this is the best one can do as a dual,  since if  $X$ was an Eilenberg-MacLane space, namely $X=K(M,n)$,  then $G$ would be isomorphic to $\Aut (M)$ which is an infinite group for $M\not=\{0\}$ and would contradict our hypothesis.

\medskip
The present work represents a further significant progress with respect to \cite{CV2} since the class of orthogonal groups for which Theorem \ref{thmgroup}  is applied to, strictly contains finite groups.  This latter approach leads to an alternate  way of tackling Kahn's realizability problem. Section \ref{sect:garibaldi}, devoted to that end, should be thought of as a first attempt to give a complete answer to the question of whether any group is the group of self homotopy equivalences of a rational space.

\medskip
Henceforth  $M$ denotes an $n$-dimensional $\Q$-module.
\section{Finite groups as orthogonal groups of algebraic forms}\label{invariant}

We use in this section Invariant Theory and we recall some pertinent facts, referring to  \cite{benson} for more information.
The ring of polynomial functions on $M$ is denoted by $ \Q [M]$,  which  by fixing a basis
$\{{v_1}, \ldots, {v_n}\}$ of  the dual $M^\ast$, can be expressed as $ \mathbb Q [{v_1}, \ldots, {v_n}]$.   An \emph{$n$-algebraic form of degree $d$} is an homogeneous polynomial function $ q (v_1, . . . ,v_n)$ of degree $d$.
We say that $f \in {\rm{GL}} (M)$ is an \emph{automorphism of the $n$-algebraic form $q$} if  $q \circ f = q$.  The set of all the automorphisms of the form $q$ is a group which we call \emph{the orthogonal group of $q$} and we denote it by $\Or (q)$. In an analogous way,  we define the \emph{orthogonal group of a family of $n$-algebraic forms} $\mathcal Q = \{q_0, \ldots, q_r\}$  by $ \Or (\mathcal Q) = \{ f \in {\rm{GL}} (M)  \mid q_i \circ f = q_i \; \text{for every}\;  i = 0, \ldots, r\}$.

If a group $G$ acts on $M$, the action of  $G$ over $M$ extends to an action  over $\Q[M]$ via $(g \cdot p) (u) = p (g^{-1} u)$ for $g \in G$, $ p \in \Q[M]$. The invariant ring  $\Q[M]^G$ is the set of all  the fixed points of this action which is the basic object of study of Invariant Theory.  A classical result, due to D. Hilbert and E. Noether, which is central to this paper, states that if $G$ is a finite group then the invariant ring is finitely generated; see \cite[Corollary 1.3.2]{benson}.
One can also look at the action of $G$ over the field of fractions $\Q (M)$  given by $g \cdot (f_1 / f_2)  = (g \cdot f_1)/ (g \cdot f_2)$. In this case, $\Q (M)$ is a Galois extension of  the fixed points of this action, $\Q (M)^G$, with Galois group $G$, and $\Q(M)^G$ is the field of fractions of the invariant ring; see \cite[Proposition 1.1.1]{benson}.

\medskip

Before proving a technical lemma, we introduce some notation.

\begin{definition}\label{def:prerealizable}
A family of $n$-algebraic forms $\mathcal Q =  \{q_0, \cdots,  q_{r+1}\} \subset \Q[v_1, \ldots, v_n]$  is realizable if:

\begin{enumerate}

\item\label{lem:cambiar_formas:2} $q_{r+1}=q_0^s$ for some $s\geq {\rm{max}}  \{n, \lceil {\frac{\deg(q_r)}{\deg(q_0)}}\rceil +1 \} $
\item\label{lem:cambiar_formas:3} $\deg(q_{i+1}) - \deg(q_{i})>1$, for $i=0,\ldots,r$.
\item\label{lem:cambiar_formas:1} $q_0=\lambda_1v_1^d + \ldots + \lambda_nv_n^d$,  for  some integer $d > 1$, $\lambda_i \neq 0$ for $i = 1, \ldots, n$.
\end{enumerate}
\vspace{0.25cm}
We say that $\mathcal Q$ is pre-realizable if $\mathcal Q$ satisfies (\ref{lem:cambiar_formas:2})  and (\ref{lem:cambiar_formas:3}).

\end{definition}

 Now, recall that the ring of polynomial functions is a  unique factorization domain. Then, we  can prove the following lemma.
\begin{lemma}\label{lem:cambiar_formas}
Let $\mathcal P = \{p_0, \ldots, p_r\}$ be an arbitrary family of $n$-algebraic forms on $M$. Then, there exists a pre-realizable family $\mathcal Q = \{q_0, \ldots, q_{r+1}\}$  of $n$-algebraic forms on $M$ such that $\Or(\mathcal P)=\Or(\mathcal Q)$ with $q_0=p_0$.

\end{lemma}
\begin{proof}
Let  $q_0$ equal $p_0$. Inductively, we define $q_i=p_i q_{i-1}  q_0$ for $i=1,\ldots,r$ and $q_{r+1}=q_0^s$ for $s \geq  {\rm{max}}  \{n, \lceil {\frac{\deg(q_r)}{\deg(q_0)}}\rceil +1 \}$. It is immediate that \emph{\eqref{lem:cambiar_formas:2}}  and \eqref{lem:cambiar_formas:3} from Definition \ref{def:prerealizable} hold for the family $\mathcal Q = \{q_0, \cdots, q_{r+1}\}$, so it only remains to prove that the orthogonal groups coincide.
 Now, by induction,  if $g\in \Or(\mathcal P)$, then $p_i  \circ g =p_i$ for all $i$, and therefore
\begin{align*}
q_i \circ g&=(p_iq_{i-1} q_0) \circ g =( p_i \circ g ) (q_{i-1} \circ g ) (q_0 \circ g)= p_iq_{i-1} q_0\\&=q_i,\text{ for } i=1,\ldots,r,\text{ and} \\
q_{r+1} \circ g &=q_0^s \circ g = (q_0 \circ g)^s=q_0^s\\&=q_{r+1},
\end{align*}
hence $g\in \Or(\mathcal Q)$. On the other hand, if $g\in \Or(\mathcal Q)$ then
\begin{align*}
p_iq_{i-1} q_0&=q_i\\
&=q_i \circ g\\
&=(p_iq_{i-1} q_0    )\circ g=(p_i \circ g) (q_{i-1} \circ g) (q_0 \circ g)  ,\text{ for } i=1,\ldots,r,\\
&=(p_i \circ g)q_{i-1}q_0,\text{ for } i=1,\ldots,r.
\end{align*}
Since $\Q[M]$ a unique factorization domain, then $p_i=p_i \circ g$, so $g\in \Or(\mathcal P)$.
\end{proof}

With these results in mind, we can prove the following result.
\begin{theorem}\label{thm:orthogonal}  Let $M$ be a faithful $n$-dimensional $\Q G$-module where $G$ is a finite group. Then, there exists a family $\mathcal Q = \{q_0,  \ldots, q_{r+1}\}$ of realizable $n$-algebraic forms over $M$ such that $\deg(q_0)=2$ and
 $G $ is the orthogonal group  $\Or(\mathcal Q) \leq {\rm{GL}} (M)$. \end{theorem}

\begin{proof}

We can consider $G \leq {\rm{GL}}  (M)$.  Now, as the invariant ring $\Q [M]^G$ is finitely generated, there exists a family of  $n$-algebraic forms $\mathcal P = \{p_1, \ldots, p_r \}$  such that if  $g \in G$  then $g \cdot p_i = p_i$ for every $i = 1, \ldots, r$.  Now,  take any $g  \in  {\rm{GL}} (M)$ such that $g \cdot p_i = p_i$ for every $i = 1, \ldots, r$. In particular,  $g $ fixes every element in the field of fractions of the invariant ring, $\mathbb Q (M)^G$. Then, as  $G = Gal (\Q(M) /\Q (M)^G)$ we conclude that $g \in G$. So, we have proved that any element $g \in {\rm{GL}} (M)$ that fixes every $n$-algebraic form in $\mathcal P$,  is in fact in $G$. In other words, we obtain that $G = \Or (\mathcal P)$.

Consider now the following positive definite form $p_0 = w_1^2 + \ldots + w_n^2$, for an arbitrary basis of $M^\ast$, 
and transform it into a $G$-invariant positive definite $n$-form as follows:
 $\displaystyle q_0 = \sum_{{{g\in G} }} g \cdot p_0 $.  Since every $n$-form of degree $2$ over $\Q$ can be diagonalized, for a basis $ { \{v_1, \ldots, v_n\} \subset M^\ast} $ that we fix,   we can write $q_0 = \lambda_1 v_1^2 + \cdots + \lambda_n v_n^2$, where $\lambda_j > 0$ for all $j$ and, by construction, it is clear that $G = \Or (q_0, p_1, \ldots, p_r  )$.  Now, by using Lemma \ref{lem:cambiar_formas},  we transform $\{q_0, p_1, \cdots, p_r\}$ into
a  pre-realizable family $\mathcal Q = \{q_0, \ldots, q_{r+1}\}$  such that $G = \Or (\mathcal Q)$, but $q_0$ already verifies condition (\ref{lem:cambiar_formas:1}) from Definition \ref{def:prerealizable} so we conclude.
\end{proof}

\begin{example}\label{example:s_n}
For $\Sigma_n$  the symmetric group on a finite set of $n$ elements, let $M$ be the permutation $\Q \Sigma_n$-module with basis $\{m_1,\ldots,m_n\}$. Then, given a permutation $\sigma\in\Sigma_n$, the action of $\sigma$ on $M$ is given by $\sigma(m_i)=m_{\sigma(i)}$ and, if $x_i\in M^*$ denotes the dual of $m_i\in M$, the invariant ring $\Q[M]^{\Sigma_n}\subset\Q[x_1,\ldots,x_n]$ is the polynomial algebra generated by the elementary symmetric functions $e_i$, $i=1,\ldots,n$, defined by the formal equation
$$f(X)=\prod_{i=1}^n(X-x_i)=X^n+\sum_{i=1}^n (-1)^ie_i(x_1, \ldots, x_n)X^{n-i}$$
where $X$ is an indeterminate, \cite[Example of  p. 2]{benson}.

Now, following the lines of the proof of Theorem \ref{thm:orthogonal}, we obtain that  $\Sigma_n=\Or({\mathcal Q})\leq {\rm{GL}}(M)$ where  $\mathcal Q = \{q_0, \cdots, q_{n+1}\}$ is the following realizable family of $n$-algebraic forms
\begin{align*}
 q_0&= n!\sum_{i=1}^n x_i^2, \\
q_j&=  e_j q_{j-1} q_0, \;  \text{for} \; j=1,\ldots,n, \\
q_{n+1}&= q_0^s,  \;  \text{for} \;  s \geq \lceil \frac{ (n+4) (n+1)}{4}\rceil +1.
\end{align*}

\end{example}

\section{Kahn's problem for orthogonal groups}\label{Model}

In this section we realize orthogonal groups of a finite family $\mathcal Q = \{q_0, \ldots, q_r\} \subset \Q [M]$ of rational $n$-algebraic forms under the hypothesis that $q_0 =\lambda_1v_1^d + \ldots + \lambda_nv_n^d$,  for  some basis $\{v_1, \ldots, v_n\}$ of the dual $M^\ast$ and some $d > 1$, with $\lambda_i \neq 0$ for $i = 1, \ldots, n$. 

 We recall some results on rational homotopy theory and refer the reader to \cite{FHT}. If $W$ is a graded rational vector space, we write $\Lambda W$ for the free commutative graded algebra on $W$, which is a symmetric algebra on $W^{\text{even}}$  tensored with an exterior algebra on $W^{\text{odd}}$. A Sullivan algebra is a commutative differential graded algebra which is free as commutative graded algebra on a simply connected graded vector space $W$ of finite dimension in each degree. It is minimal if in addition $\di(W) \subset \Lambda^{\geq 2}W$.  The geometric realization funtor of Sullivan, establishes an isomorphism between finite-type minimal Sullivan algebras and simply-connected spaces, and an isomorphism between $W$ and $\pi_\ast (X)$ as graded modules.

\begin{definition}\label{defmodel}
Let $\mathcal Q = \{q_0, \ldots, q_{r+1}\} \subset \mathbb Q [v_1, \ldots, v_n]$ be a family of realizable $n$-algebraic forms, where $\{v_1, \ldots, v_n\}$ is a basis of $M^\ast$ and $ d = \deg (q_0) \geq 2$. Then,
for any $k$ integer satisfying $\deg (q_{r+1}) < 2k + (d-1)$,  we define the minimal Sullivan model
$$\M_{(\mathcal Q, k)}= \Big(\Lambda(x_1,x_2,y_1,y_2,y_3, z, v_j \mid  j = 1, \ldots, n), \di \Big)$$
 where degrees and differential are described by
 \begin{alignat*}{2}
&\vert x_1 \vert= 8, \qquad \qquad&\di(x_1)&=0\\
&\vert x_2 \vert= 10,& \di(x_2)&=0\\
&\vert y_1\vert = 33,& \di(y_1)&=x_1^3x_2\\
&\vert y_2\vert= 35,& \di(y_2)&=x_1^2x_2^2\\
&\vert y_3\vert = 37,& \di(y_3)&=x_1x_2^3\\
&\vert v_j \vert= 40,& \di(v_j)&=0\\
&\vert z \vert = 80k + 40 d -41,\qquad & \di(z)&=\displaystyle \sum_{i=1}^{r+1} q_i x_1^{10k + 5 (d-1)- 5\deg( q_i)} + q_0 (x_1^{10k-5} + x_2^{8k-4})\\
& & &+ x_1^{10k + 5 (d-4)} (y_1y_2 x_1^4x_2^2 - y_1y_3x_1^5x_2 + y_2y_3x_1^6)\\
& & & + x_1^{10k +5(d-1)} + x_2^{8k+4(d-1)}.
\end{alignat*}
\end{definition}

We can now prove one of our main results. For clarity purposes,  we include some technical lemmas at the end of this paper,  in Section \ref{lemma}.
\begin{theorem}\label{thmgroup}
Let  $ \mathcal M_{(\mathcal Q, k)}$ be the minimal Sullivan model from above.
Then,  the group of self homotopy equivalences
$\E (\mathcal M_{(\mathcal Q, k)})$ is isomorphic to the orthogonal group
$\Or (\mathcal Q)$.
\end{theorem}

\begin{proof}
For any  $ g \in \Or (\mathcal Q) $, we can define
$f_g \in \Aut(\M_{(\mathcal Q, k)})$ given by:
$$\begin{alignedat}{2}
f_g(x_i) =&  x_i \\
f_g(y_i) =& y_i, \\
 f_g(v_j) =&  g \cdot v_j \\
  f_g(z) = &    z  \\
  \end{alignedat}$$
Hence $\Or (\mathcal Q)  \leq \Aut (\M_{(\mathcal Q, k)})$ and also $\Or (\mathcal Q)  \leq \E (\M_{(\mathcal Q, k)}).$ Let us now  consider an arbitrary $f \in \Aut (\M_{(\mathcal Q, k)})$. By degrees reasoning, we have:
\begin{equation}\label{description-f}
\begin{alignedat}{2}
f(x_1) =&  a_1x_1 \\
f(x_2) =& a_2x_2  \\
 f(y_i) =&  b_iy_i \\
f(v_j) = & A (v_j) + a_1 (j) x_1^5 + a_2 (j) x_2^4, \, \, \text{where} \, A \in {\rm{GL}}(M^\ast)\\
f(z) = & cz + y_1A_1 + y_2A_2 + y_3A_3 + y_1y_2y_3D, \,\,  \text{where}  \, A_1, A_2, D \in \mathbb Q[x_1, x_2, v_1  \ldots, v_n] \\
\end{alignedat}
\end{equation}

 Now, using that $f(\di y_i) = \di f(y_i) $, for $i=1, 2, 3$ we obtain:
\begin{equation}\label{coefficients}
\begin{alignedat}{2}
b_1 =&  a_1^3a_2 \\
b_2 =& a_1^2a_2^2 \\
b_3 = &a_1a_2^3\\
\end{alignedat}
\end{equation}

We now compare $f(\di z)$ to $\di f(z)$ which must be equal:
\begin{equation}\label{dz}
\begin{alignedat}{2}
f(\di(z))=&  \sum_{i=1}^{r+1} f(q_i) (a_1x_1)^{10k + 5(d-1) - 5\deg(q_i)} + f(q_0)  \Big((a_1x_1)^{10k-5} + (a_2x_2)^{8k-4} \Big)\\
&+  (a_1x_1)^{10 k + 5 (d-4)} \Big((b_1y_1)(b_2y_2) (a_1x_1)^4(a_2x_2)^2 - (b_1y_1)(b_3y_3)(a_1x_1)^5(a_2x_2) \\
&+ (b_2y_2)(b_3y_3)(a_1x_1)^6\Big) + (a_1x_1)^{10k +5 (d-1)} + (a_2x_2)^{8k+4(d-1)}\\
\end{alignedat}
\end{equation}

\begin{equation}\label{df}
\begin{alignedat}{2}
\di f(z)=&c \Big [\sum_{i=1}^{r+1} q_i x_1^{10k + 5 (d-1) - 5\deg( q_i)} + q_0 (x_1^{10k-5} + x_2^{8k-4})  \\
&+ x_1^{10 k + 5 (d-4)} (y_1y_2 x_1^4x_2^2 - y_1y_3x_1^5x_2 + y_2y_3x_1^6)
+ x_1^{10k +5(d-1)} + x_2^{8k+4(d-1)} \Big] \\
& +x_1^3x_2A_1 + x_1^2x_2^2A_2 + x_1x_2^3A_3 + x_1^3x_2y_2y_3D - x_1^2x_2^2y_1y_3D + x_1x_2^3y_1y_2D\\
\end{alignedat}
\end{equation}

We first observe that $D= 0$.  This is easily obtained since in (\ref{df}) there are  $x_1^3x_2y_2y_3D$ like terms  while in (\ref{dz}) no term in $y_2y_3$ contains $x_2$ as a factor.

We now claim that $a_2 (j) = 0,$ for $j = 1, \ldots,n$.
In (\ref{dz}) the term $f(q_0)(a_2x_2)^{8k-4}$ contains the expression
\begin{equation}\label{a2j}
\displaystyle \lambda_j \sum_{j= 1}^n  \binom{d}{1} A (v_j) (a_2(j)x_2)^{4(d-1)}(a_2 x_2)^{8k-4}.
\end{equation}
Using again that $f(\di (z))$ and $\di (f(z))$ must be equal and,  since in  (\ref{df})  the only possible summands in $A(v_j)x_2^{8k+4(d-2)}$ arise from those which contain $A_i, \, i=1,2, 3$, which also contain $x_1$ as a factor,  we deduce that the expression (\ref{a2j}) is zero. This directly implies that $$\displaystyle \sum_{j= 1}^n  \lambda_j d  a_2^{8k-4} a_2(j)^{4(d-1)}A (v_j) = 0.$$
Now, $\{A(v_j)\}_{j=1}^n$ are  linearly independent,  since $A \in {\rm{GL}} (M^\ast)$ and $\{v_j\}_{j=1}^n$ are  linearly independent.  Therefore
 $\lambda_j d a_2^{8k-4} a_2(j)^{4(d-1)} = 0$ for every $j = 1, \ldots, n$, and since $a_2^{8k-4} \in \Q^\ast$, $d>1$, and $\lambda_i \neq 0$, we get that $a_2(j) =0, j= 1, \ldots, n$.

Next we prove that $a_1(j)=0$, for $j = 1, \ldots, n$.
To that end, we consider  monomials in $\di(z)$ containing both $x_1$ with the least possible exponent,  and containing \emph{all} the variables $v_1, \ldots, v_n$. Observe that those monomials exclusively come from $q_{r+1} x_1^{10k +5(d-1) - 5 \deg (q_{r+1})}$, and  their existence is guaranteed since $\mathcal Q$ is a realizable family of forms which in particular implies that $q_{r+1} = q_0^s$ for an adequate $s$.   Now, we order them by lexicographic order, that is, $v_1^{d_1} \cdots v_n^{d_n} x_1^{10k +5(d-1)- 5 \deg (q_{r+1})}>v_1^{e_1} \cdots v_n^{e_n} x_1^{10k +5 (d-1)- 5 \deg (q_{r+1})}$  if $d_1 > e_1$ or if $d_1 = e_1$, $d_2 > e_2$,  and so on. We pick the highest of  those monomials in $\di(z)$ which, up to some non-zero coefficient,  has the form:
 $$v_1^{\alpha_1} \cdots v_n^{\alpha_n} x_1^{10k +5(d-1) - 5 \deg (q_{r+1})}$$
Therefore, there exists a monomial in (\ref{dz}),  up to some coefficient in the variables $A (v_i)$'s,
\begin{equation}\label{Avi}
\binom{\alpha_n}{1} a_1(n)A(v_1)^{\alpha_1} \cdots A(v_n)^{\alpha_n - 1} x_1^{10k +5d - 5 \deg (q_{r+1})}.
\end{equation}
Due to the exponent of $x_1$, the monomial (\ref{Avi}) does not appear in (\ref{df}) (recall that $\deg (q_i) - \deg (q_j) >1$ for every $i \neq j$). Therefore it must be either zero or it must be cancelled by other monomial in (\ref{dz}). Observe that it cannot be cancelled from monomials coming from $q_j$, with $j < r+1$,  as we mentioned above, so  the only possibility for it to be cancelled is to be the image of a monomial $v_1^{\alpha_1}\cdots v_{i_0}^{\alpha_{i_0+1}} \cdots v_n^{\alpha_n-1}x_1^{10k+5(d-1)-5 \deg( q_{r+1})}$.   Now, that monomial cannot appear in $q_{r+1}$ because it is strictly bigger than
$v_1^{\alpha_1} \cdots v_n^{\alpha_n} x_1^{10k +5 (d-1) - 5 \deg (q_{r+1})}$, which by hypothesis was the highest monomial in $q_{r+1}$ of that form. Therefore we conclude that (\ref{Avi}) is zero, and consequently $a_1(n) = 0$.  The same arguments, reordering the $v_i's$  successively, show that $a_1 (j) = 0$, for $j = 1, \dots, n-1$.

We have mentioned before that  every $A_i$-term  in (\ref{df}),  $i=1,2,3$,  has a $x_1x_2$ factor but it does not have a $y_1y_2$ factor, whereas in (\ref{dz}) every summand with a $x_1x_2$ factor has also a $y_1y_2$ factor. Therefore, $$0 =x_1^3x_2A_1 + x_1^2x_2^2A_2 + x_1x_2^3A_3 = \di(y_1A_1 + y_2A_2 + y_3A_3) $$ implying that $y_1A_1 + y_2A_2 + y_3A_3 = dm_z$ for some element $m_z$, by Proposition \ref{dmz}.

Observe that from these conclusions above,  the third line in  (\ref{df}) is zero. We compare now the second line from (\ref{df})  to the second and third lines in (\ref{dz}) and we get
\begin{equation}\label{c}
\begin{alignedat}{2}
c=& (a_1)^{10n-6}(a_2)^{2}b_1b_2\\
=&(a_1)^{10n-5}a_2b_1b_3\\
=&(a_1)^{10n-4}b_2b_3\\
=&(a_1)^{10n+5}\\
=&(a_2)^{8n+4}
\end{alignedat}
\end{equation}
and, using equations from (\ref{coefficients}), we also get that
$$c = a_1^{10n-1}a_2^{5}  =a_1^{10n+5} = a_2^{8n+4} ,$$
so that $a_1^{2n+1} = 1$, and therefore $a_1 = 1 = a_2 = b_1 = b_2 = b_3 = c =1 $.

We finally compare the first line of (\ref{dz}) to the first line of (\ref{df}). Since
the $n$-algebraic forms $\{q_0, \ldots, q_{r+1}\}$ verify  $2< \deg(q_0) < \deg (q_1)  < \ldots < \deg (q_{r+1}) $, it is straightforward that $f(q_i) = q_i $ for every $i= 0, \ldots, r+1$. Recall that the polynomial expression $q_i$  comes from the $n$-algebraic form $q_i \in \Q [M^\ast]$ and that $\M_{(\mathcal Q, k)}^{40} = M^\ast$. So as $f_{\mid M^\ast}  = A \in {\rm{GL}} (M^\ast) $, then $ q_i = A  (q_i) = q_i \circ A^t$ for every $i = 0, \ldots, r+1$, which implies that $A^{t} \in \Or (\mathcal Q)$. 

 Gathering altogether, we have proved that there exists one element $g= {(A^{t})}^{-1} \in \Or (\mathcal Q)$ such that $f = f_g + dm_z $. Since $f$ is homotopically equivalent to  $f_g$,  we conclude our proof.
\end{proof}

\begin{corollary}\label{diagonal}  Let $\Or (p_0, \ldots, p_r)  $ be the orthogonal group of a family of $n$-algebraic forms over an $n$-dimensional $\Q$-module,  such that $p_0$ is diagonalizable over $\mathbb Q$ with all the $\lambda_i \neq 0$ and $\deg (p_0) >1$. Then, $\Or (p_0, \ldots, p_r) $ can be realized by infinitely many non homotopically equivalent (rational) spaces.
\end{corollary}
\begin{proof} By hypothesis $q_0 = \lambda_1v_1^d + \ldots + \lambda_nv_n^d$ for a given basis that we fix. By Lemma \ref{lem:cambiar_formas},  there exists a family of pre-realizable (realizable under the assumptions) $n$-algebraic forms  $ \mathcal Q = \{q_0, \ldots, q_{r+1}\}$  such that $ \Or (\mathcal Q) = \Or (p_0, \ldots, p_r) $.  Then,  by Theorem \ref{thmgroup} there exists, for every  integer $k$  s.t.\ $\deg (q_{r+1}) < 2k + (\deg q_0 -1)$,  a minimal Sullivan model $\mathcal M _{(\mathcal Q, k)}$ that realizes $\Or (\mathcal Q) $. It is clear that if $k $ is different from $k'$, then $\mathcal M _{(\mathcal Q, k)} $ and $\mathcal M _{(\mathcal Q, k')}$ are non homotopically equivalent.
\end{proof}

\section{Linear algebraic groups as orthogonal groups of algebraic forms}\label{sect:garibaldi}

This section is a first attempt to  give a complete answer to Kahn's realizability problem. Before that, we take the opportunity to recall what we have achieved in the previous sections. In Section \ref{invariant}  we proved that any finitely generated $\Q G$-faithful module $M$, provided that $G$ is finite, is the orthogonal group $\Or (q_0, \ldots, q_r) \leq {\rm{GL}} (M)$ of a realizable family of algebraic forms over $M$, Theorem \ref{thm:orthogonal}. Recall that realizable essentially means that $q_0$ is diagonalizable over $\mathbb Q$ with all the coefficients non zero. Within that framework, the  theorem above implies that realizing group actions is  equivalent to realizing orthogonal groups in the sense of  Kahn. Therefore, bringing back techniques from \cite{CV2}, in Section \ref{Model} we showed how to realize any orthogonal group  of a realizable family of rational algebraic forms,  Theorem \ref{thmgroup}.

This latter approach leads to an alternate  way of tackling Kahn's realizability problem that we would like to exploit. The first thing to be aware of is that our techniques involve finitely generated Sullivan algebras over $\mathbb Q$, which in particular implies that their group of self homotopy equivalences is a linear algebraic group  defined over $\Q$ \cite[Theorem 6.1]{Su}.  Therefore, the following question can be raised:
\begin{question} Is every linear algebraic group defined over $\mathbb Q$ the orthogonal group of a realizable family of rational algebraic forms?
\end{question}

We point out that most of the existing literature on linear algebraic groups is developed over $\C$ so we need to introduce some specific notation. Given an $n$-dimensional $\Q$-module $M$,  one refers to the group of matrices with entries in $\C$ as the $\C$-points of ${\rm{GL}} (M)$.  An \emph{algebraic linear group} $G$  \emph{over}  $\Q$ is determined by any subgroup of the $\C$-points of ${\rm{GL}}(M)$ defined by polynomial equations in the entries where the coefficients of the polynomials lie in $\Q$.
We denote by $M_{\C}=M\otimes_\Q \C$ the complexification of $M$. A family of rational algebraic forms $\mathcal Q = \{q_0, \ldots, q_r\}\subset \Q[M]$ gives rise to a family of complex algebraic forms (defined over $\Q$) that we also denote by ${\mathcal Q}\subset \C[M_\C]$; we write $\Or_\C (\mathcal Q) = \{ f \in {\rm{{\rm{GL}}} }(M_\C)  \mid q_i \circ f = q_i \; \text{for every}\;  i = 0, \ldots, r\}$ for the orthogonal group. Then, $\Or({\mathcal Q})$ equals $\Or_\C({\mathcal Q})(\Q)$, the group of rational points of $\Or_\C({\mathcal Q})$.

\medskip
The following argument shows that when we consider orthogonal groups of a family of realizable forms, we are indeed considering orthogonal representations in the classical sense (i.e.\ preserving a quadratic form).

\begin{lemma}\label{lem:quadrat}
Let $M$ be an $n$-dimensional $\Q$-module and let $\mathcal Q = \{q_0, \ldots, q_{r+1}\} \subset \mathbb Q [M]$ be a family of realizable algebraic $n$-forms with $ d = \deg (q_0) >1.$ If $\Or_\C({\mathcal Q})$ is infinite, then $d=2.$
\end{lemma}
\begin{proof}
Assume $d>2$ and  let  $q_0 =\lambda_1v_1^d + \ldots + \lambda_nv_n^d$ for some basis $\{v_1,\ldots, v_n\}$ of the dual  $M^\ast$, where $\lambda_i\ne 0$ for $i=1,\ldots, n$. Let $(q_0)_i$ denote the $i$-th partial derivative of $q_0$, that is, $(q_0)_i=d\lambda_iv_i^{d-1}.$ Then $(0,\ldots,0)$ is the only common zero in $M_\C$ of the polynomials $(q_0)_1,\ldots,(q_0)_n,$ and therefore $q_0$ is nondegenerate in the sense of \cite[\S 1]{Orlik-Solomon}. Therefore, according to \cite[Theorem 2.1]{Orlik-Solomon} the group $\Or_\C(q_0)$ is finite (recall $d>2$). Now, since $\Or_\C({\mathcal Q})$ is infinite and $\Or_\C({\mathcal Q})\subset \Or_\C(q_0)$, we get a contradiction.
\end{proof}

\begin{remark}
In view of Theorem \ref{thm:orthogonal} and Lemma \ref{lem:quadrat}, given an $n$-dimensional $\Q$-module $M$ and a family of realizable $n$-algebraic forms $\mathcal Q = \{q_0, \ldots, q_{r+1}\} \subset \mathbb Q [M]$, there exists a family of realizable $n$-algebraic forms $\mathcal P = \{p_0, \ldots, p_{s+1}\} \subset \mathbb Q [M]$ such that $\deg (p_0) =2$ and $\Or_\C({\mathcal Q})=\Or_\C({\mathcal P})$ (and therefore $\Or({\mathcal Q})=\Or({\mathcal P})$).
\end{remark}

The question of whether it is possible to describe algebraic linear groups as orthogonal groups of algebraic forms has been addressed by Garibadi-Guralnick \cite{Garibaldi-Guralnick}  when the group is simple; see also \cite{Bermudez-Ruozzi}. The following result is a restatement of \cite[Theorem 6.6]{Garibaldi-Guralnick}, and we sketch here the proof to emphasize that the algebraic forms involved can be chosen to be defined over $\Q$.

\begin{theorem}\label{GG}
Let $G$ be a center free simple linear algebraic group defined over $\Q$ which is not $C_2$ type. Then there exist a $\Q$-module $M$ of dimension $n=\dim\big(\operatorname{Lie}(G)\big)$ and a family of realizable $n$-algebraic forms $\mathcal Q = \{q_0, q_1, q_2\} \subset \mathbb Q [M]$ such that the following hold
\begin{itemize}
\item[i)] $\deg (q_0)=2,$
\item[ii)] $G\subset \Or_\C({\mathcal Q})\subset {\rm{GL}}(M_\C)$, and this representation is equivalent to the adjoint representation $G\subset {\rm{GL}}(\operatorname{Lie}_\C(G)),$
\item[iii)] $[\Or_\C({\mathcal Q}) : G]$ is finite.
\end{itemize}
\end{theorem}
\begin{proof}
The proof is along the lines in \cite[Section 6]{Garibaldi-Guralnick}. Let $\mathcal R$ be a root system underlying $\operatorname{Lie}(G),$ and define $N={\mathcal P}^\vee\otimes \Q$ where $\mathcal P^\vee$ denotes the dual of the root lattice $\mathcal R^\vee$. If $W$ is the Weyl group of $G$, then $W$ acts on $N$ and $\Q[N]^W$ is a polynomial algebra with generators $p_1,\ldots,p_l$ such that the natural map $\Q[N]^W\otimes\C\to\C[N_\C]^W$ is an isomorphism where $\C[N_\C]^W$ is the polynomial algebra generated by the images of $p_1,\ldots,p_l$ \cite[Lemma 6.3]{Garibaldi-Guralnick}. Moreover, there is a natural map $\C[\operatorname{Lie}_\C(G)]^G\to \C[N_\C]^W$ which is an isomorphism, so we can choose $f_1,\ldots f_l$ homogeneous polynomials defined over $\Q$ by pulling back the polynomials $p_1,\ldots,p_l$ above, such that they generate $\C[\operatorname{Lie}_\C(G)]^G\to \C[N_\C]^W$. We define $M$ to be dual of the $\Q$-module spanned by the variables of $p_i$, so $M_\C=\operatorname{Lie}_\C(G)$ and $\C[M_\C]=\C[\operatorname{Lie}_\C(G)]$.

Let $f_1$ be (a multiple of) the quadratic Killing form on $\operatorname{Lie}_\C(G)$, and let $f$ be one of the higher degree invariants $f_2,\ldots,f_l$, which is not a power of $q_0$ (take $f=f_1$ if $G$ is of type $A_1$). Then, define $q_0=f_1$, $q_2=q_0f$ and $q_3$ an appropriate power of $q_0$ according to Definition \ref{def:prerealizable}.(3) so that $\mathcal Q = \{q_0, q_1, q_2\} \subset \mathbb Q [M]$ becomes a family of realizable algebraic forms. Now,  according to \cite[Theorem 6.6]{Garibaldi-Guralnick} (and the related discussion in case $G$ is $A_1$ type), $G$ is the connected component of $\Or_\C(f)$, that is, $[\Or_\C(f) : G]$ is finite. Therefore, since $G\subset\Or_\C({\mathcal Q})\subset \Or_\C(f)$, then $[\Or_\C({\mathcal Q}) : G]$ is finite too.
\end{proof}
\begin{remark}
Observe that, since the adjoint representation is faithful only if $G$ is center free,  that condition can not be avoided from the previous theorem.
\end{remark}

In general, it is not easy to provide an explicit description of the algebraic forms appearing in Theorem \ref{GG} (see \cite[Example 6.7]{Garibaldi-Guralnick} for  $G=E_8$). Nevertheless, it is possible to construct algebraic forms that fit in this setting:
\begin{example}\label{example:e8}
 Let $G=E_8$. Since $n=\dim\big(\operatorname{Lie}(E_8)\big)=248$,  we have to consider algebraic $248$-forms. Let $f_1$ be a $G$-invariant quadratic algebraic form defined over $\Q$ as in the proof of Theorem \ref{GG}, and let $f$ be the polynomial of degree $8$ described in \cite[Equation (2.3)]{CeP}.  Then $\mathcal Q = \{q_0, q_1, q_2\}$, where $q_0=f_1$, $q_1=q_0 f$ and $q_2=q_0^{248}$, is a realizable family of algebraic forms such that $\Or_\C({\mathcal Q})=E_8\times {\mathbb Z}/2$.

 Indeed, according to \cite[Theorem 3.1]{Garibaldi-Guralnick}, $\Or_\C(f)=E_8\times {\mathbb Z}/8$ where the factor ${\mathbb Z}/8$ is generated by diagonal matrices of the form $\omega {\mathbb I}$ with $\omega^8=1$. But since $f_1$ is a non trivial quadratic polynomial, a diagonal matrix $\omega {\mathbb I}$ stabilizes $f_1$, if and only if, $\omega^2=1$. Therefore $\Or_\C({\mathcal Q})=\Or_\C(f_1,f)=E_8\times {\mathbb Z}/2$, where the factor ${\mathbb Z}/2$ is generated by diagonal matrices of the form $\omega{\mathbb I}$ with $\omega^2=1$.
\end{example}

Combining Theorems \ref{thmgroup} and \ref{GG} we obtain the following.

\begin{corollary}\label{Kan_for_algebraic}
Let $G$ be a center free simple linear algebraic group defined over $\Q$ which is not $C_2$ type. Then there exists a rational space $X$ such that $G(\Q)$, the group of rational points of $G$, is a subgroup of $\E(X)$ and $[\E(X):G(\Q)]$ is finite.
\end{corollary}

The following example illustrates that apart from the adjoint representation in Theorem \ref{GG}, other representations can be used to describe a simple linear algebraic group as the orthogonal group of a family of realizable algebraic forms, leading to refinements of Corollary \ref{Kan_for_algebraic}.

\begin{example}\label{example:g2}
Let $G=G_2$, and let $M=\Q^{\oplus 7}$. If we fix $\Q[M]=\Q[x_0, x_i,x'_i\colon i=1,2,3]$ then according to \cite[pag.\ 194]{aschbacher} $G_2=\Or_\C(f_0,f_1)$ where $f_0$ is the quadratic form $$f_0=-2x^2_0 +x_1x'_1+x_2x'_2+x_3x'_3,$$
and $f_1$ is Dickson alternating trilinear form
$$f_1=x_0x_1x'_1 + x_0x_2x'_2 + x_0x_3x'_3 + x_1x_2x_3 + x'_1x'_2x'_3.$$
Considering the change of variables $v_1=x_0$, and $v_{2i}=x_i+x'_i$, $v_{2i+1}=x_i-x'_i$ for $i=1,2,3$, we obtain in $\Q[M]=\Q[v_i\colon i=1,\ldots,7]$ the expressions
$$f_1=\frac{1}{4}\Big(v_1(v_2^2-v_3^2+v_4^2-v_5^2+v_6^2-v_7^2)+v_2v_5v_7+v_3v_4v_7+ v_3v_5v_6 + v_2v_4v_6\Big)$$ and $$f_0=-2v_1 +\frac{1}{4} v_2^{2} -\frac{1}{4} v_3^2 +\frac{1}{4}v_4^{2} -\frac{1}{4} v_5^{2} +\frac{1}{4} v_6^{2}-\frac{1}{4} v_7^{2},$$ that give us a family of realizable $7$-forms ${\mathcal Q}=\{q_0, q_1, q_2\}$, where $q_0=f_0$, $q_1=f_1$, and $q_2=f_0^7$, such that $G_2=\Or_\C({\mathcal Q})$. Therefore, according to Theorem \ref{thmgroup}, there exists a space $X$ such that $\E(X)=\Or({\mathcal Q})=G_2(\Q)$.
\end{example}

\section{Lemmas about elements}\label{lemma}

We collect here technical results on the algebraic structure of the Sullivan model $\M_{(\mathcal Q, k)}$ introduced in Definition \ref{defmodel}. They are needed in the proof of Theorem \ref{thmgroup}. 

\medskip

In what follows $P = \Q [x_1, x_2, v_j : j = 1, \ldots, n].$
\begin{lemma}{\label{lemma1}} Let $A_1 \in  P^{80k+40d -74}$, for $d>1, \, k>1$. Then $A_1 = x_1^2x_2^3B_1$ with $B_1 \in P$.
\end{lemma}

\begin{proof}
Let $a,b,c $ be non negative integers such that $x_1^a x_2^bp \in P^{80k+40d -74}$ where $p \in \Q[v_j]$ is an algebraic form of degree $c$. Then $8a + 10b + 40 c = 80k+40d -74$ which implies $8a \equiv 6 \pmod {10}$ and hence, $a \equiv 2, 7 \pmod{10}$.
Then, since $a \geq 2$, $x_1^a x_2^bp  =x_1^2 (x_1^{\tilde a} x_2^bp)$ with $(x_1^{\tilde a} x_2^bp) \in P^{80k +40 d-90}$. Simplifying notation, we write $A_1 =x_1^2 \widetilde A_1$, where $\widetilde A_1 \in P^{80k +40 d-90}$.
Now, let $x_1^{\tilde a}x_2^bp \in P^{80k +40 d-90 } $. Then, $8\tilde a + 10 b + 40 c = 80k +40 d-90 $, which implies $2b  \equiv 6 \pmod{8}$ and hence $b \equiv 3, 7 \pmod{8}.$ In particular, $b \geq 3$ so $\widetilde A_1 = x_2^3B_1$, where $B_1 \in P^{80k+40d-120}.$
\end{proof}

\begin{lemma}{\label{lemma2}} Let $A_2 \in  P^{80k+40d-76}$, for $d>1, \, k>1$. Then $A_2 = x_1^3x_2^2B_2$ with $B_2 \in P$.
\end{lemma}
\begin{proof}
Let $a,b,c $ be non negative integers such that $x_1^a x_2^bp \in P^{80k+40d-76}$ where $p \in \Q[v_j]$ is an algebraic form of degree $c$. Then $8a + 10b + 40 c = 80k+40d-76$ which implies $8a \equiv 4 \pmod {10}$ and hence, $a \equiv 3, 8 \pmod{10}$.
Then, since $a \geq 3$, $x_1^a x_2^bp  =x_1^3 (x_1^{\tilde a} x_2^bp)$ with $(x_1^{\tilde a} x_2^bp) \in P^{80k+40d -100}$. Simplifying notation, we write $A_2 =x_1^3 \widetilde A_2$, where $\widetilde A_2 \in P^{80k+40d -100}$.
Now, let $x_1^{\tilde a}x_2^bp \in P^{80k + 40d -100} $. Then, $8\tilde a + 10 b + 40 c = 80k + 40d -100$, which implies $2b  \equiv 4 \pmod{8}$ and hence $b \equiv 2, 6\pmod{8}.$ In particular, $b \geq 2$ so $\widetilde A_2 = x_2^2B_2$, where $B_2 \in P^{80k+40d -120}.$
\end{proof}

\begin{lemma}{\label{lemma3}} Let $A_3 \in  P^{80k + 40d - 78}$, for $d>1, \, k>1$. Then $A_3 = x_1^4x_2B_3$ with $B_3 \in P$.
\end{lemma}
\begin{proof}
Let $a,b,c $ be non negative integers such that $x_1^a x_2^bp \in P^{80k + 40d - 78}$ where $p \in \Q[v_j]$ is an algebraic form of degree $c$. Then $8a + 10b + 40 c = 80k + 40d - 78$ which implies $8a \equiv 2 \pmod {10}$ and hence, $a \equiv 4, 9 \pmod{10}$.
Then, since $a \geq 4$, $x_1^4 x_2^bp  =x_1^4 (x_1^{\tilde a} x_2^bp)$ with $(x_1^{\tilde a} x_2^bp) \in P^{80k+40d -110}$. Simplifying notation, we write $A_3 =x_1^4 \widetilde A_3$, where $\widetilde A_3 \in P^{80k+40d -110}$.
Now, let $x_1^{\tilde a}x_2^bp \in P^{80k+40d -110} $. Then, $8\tilde a + 10 b + 40 c = 80k+40d -110$, which implies $2b  \equiv 2 \pmod{8}$ and hence $b \equiv 1, 5 \pmod{8}.$ In particular, $b \geq 1$ so $\widetilde A_3 = x_2B_3$, where $B_3 \in P^{80k+40d-120}.$
\end{proof}

\begin{proposition}\label{dmz} Let $y_1A_1 + y_2A_2 + y_3A_3 \in \mathcal M^{80k+40d -41}$ such that $A_i \in P$ and $\di(y_1A_1 + y_2A_2 + y_3A_3) = 0$. Then, there exists one element $m \in \mathcal M^{80k+40d-42}$ such that $y_1A_1 + y_2A_2 + y_3A_3 = \di (m).$
\end{proposition}

\begin{proof}
Using Lemma \ref{lemma1}, Lemma \ref{lemma2} and Lemma \ref{lemma3}, we know that $A_1 = x_1^2x_2^3B_1$, $A_2 = x_1^3x_2^2B_2$ and $A_3 = x_1^4x_2B_3$.  A straithforward computation shows that
 $ \di (y_1A_1 + y_2A_2 + y_3A_3) = x_1^5x_2^4 (B_1 + B_2 + B_3). $ As by hypothesis, $\di (y_1A_1 + y_2A_2 + y_3A_3)  = 0$,
we obtain that $ B_1 + B_2 + B_3 = 0$.

Now

\begin{equation}\label{dzz}
\begin{alignedat}{2}
y_1A_1 + y_2A_2 + y_3A_3&=\\
&= y_1(x_1^2x_2^3(-B_2-B_3))+y_2(x_1^3x_2^2B_2)+y_3(x_1^4x_2B_3)  \\
&=  x_1^3x_2y_2(x_2B_2) - y_1x_1^2x_2^2(x_2B_2) +x_1^3x_2y_3(x_1B_3) - y_1x_1x_2^3(x_1B_3)  \\
& =  \di(y_1y_2(x_2B_2)) + \di(y_1y_3(x_1B_3))\\
\end{alignedat}
\end{equation}

Hence $m = y_1y_2(x_2B_2) + y_1y_3(x_1B_3)$

\end{proof}

\noindent{{\bf{Acknowledgements.}} The authors would like to express their gratitude to Professors Skip Garibaldi, Bob Guralnick and Hiroo Shiga for helpful comments and their interest on the subject.



\begin{thebibliography}{99}

\bibitem{Aguade} J.\ Aguad\'e, \emph{Realizability of cohomology algebras: a survey}, Publ.\ Sec.\ Mat.\ UAB \textbf{26} (1982), 25--68.

\bibitem{KasperJesper} K.\ Andersen, J. \ Grodal, \emph{The Steenrod problem of realizing polynomial cohomology rings}, J.\ Topology \textbf{1} (2008), 747--760.

\bibitem{aschbacher} M.\ Aschbacher, \emph{Chevalley groups of type $G_2$ as the group of a trilinear form}, J.\ Algebra \textbf{109} (1987), 193--259.

\bibitem{benson} D.\ J.\ Benson, Polynomial invariants of finite groups. London Mathematical Society Lecture Note Series, \textbf{190}. Cambridge University Press, Cambridge, 1993. x+118 pp.

\bibitem{Bermudez-Ruozzi} H.\ Bermudez, A.\ Ruozzi, \emph{Classifying forms of simple groups via their invariant polynomials}, J.\ Algebra \textbf{424} (2015), 448--463.

\bibitem{GC} G.\ Carlsson, \emph{A counterexample to a conjecture of Steenrod}, Invent.\ Math.\  \textbf{64}, (1981), 171--174.

\bibitem{CeP} M.\ Cederwall, J.\ Palmkvist, \emph{The octic $E_8$ invariant}, J.\ Math.\ Phys.\ \textbf{48} (2007), 073505, 7pp.

\bibitem{CV2} C.\ Costoya, A.\ Viruel, \emph{Every finite group is the group of self homotopy equivalences of an elliptic space}, Acta Mathematica \textbf{213}, (2014), 49--62.

\bibitem{CV3} C.\ Costoya, A.\ Viruel, \emph{Faithful actions on commutative differential graded algebras and the group isomorphism problem},  Q.\ J.\ Math. \textbf{65} (2014), 857--867.

\bibitem{FHT} Y.\ F\'elix, S.\ Halperin, J.C.\ Thomas,  Rational homotopy theory. Graduate Texts in Mathematics, \textbf{205}, Springer-Verlag, New York, 2001.

\bibitem{Frucht1} R.\ Frucht, \emph{Herstellung  von Graphen mit vorgegebener abstrakter Gruppe},
Compositio Math.\ \textbf{6} (1939), 239--250.

\bibitem{Garibaldi-Guralnick} S.\ Garibaldi, B.\ Guralnick, \emph{Simple groups stabilizing polynomials},
Forum of Mathematics Pi \textbf{3} (2015), e3, 41 pp.

\bibitem{ka1} D.\ Kahn, \emph{Realization problems for the group of homotopy classes of self-equivalences},
Math.\ Annal.\ \textbf{220} (1976), 37--46.

\bibitem{ka3} D.\ Kahn, \emph{Some research problems on homotopy-self-equivalences, in: Groups of self-equivalences and related topics},
Lecture Notes in Math.\ \textbf{1425}, Springer, Berlin, (1990), 204--207.

\bibitem{Lashof} R.\ Lashof, \emph{Problems in Differential and Algebraic Topology.  Seattle Conference on Algebraic Topology, 1963}, Ann.\ of Maths.\ (2) \textbf{81} (1965), 565--591

\bibitem{Orlik-Solomon} P.\ Orlik, and L.\ Solomon, \emph{Singularities. II. Automorphisms of forms}, Math.\ Ann.\ \textbf{231} (1977/78), 229--240.

\bibitem{Su} D.\ Sullivan, \emph{Infinitesimal computations in topology}, Inst.\ Hautes \'Etudes Sci.\ Publ.\ Math.\ \textbf{47} (1977), 269--331.

\end{thebibliography}
\end{document}